\documentclass[12pt]{article}
\usepackage{amsmath,amsthm,amssymb,times,titlesec}
\usepackage{graphicx}
\usepackage{pdfsync}
\usepackage{float}
\usepackage[utf8]{inputenc}
\numberwithin{equation}{section}

\textheight 
            650pt
\textwidth 
           460pt

\titleformat{\subsubsection}
{\normalfont\fontsize{13}{14}\selectfont\itshape}{\thesubsubsection}{1em}{}

\newcommand\scaleddot{\scalebox{.89}{.}}
\makeatletter
\renewcommand{\ddddot}[1]{%
  {\mathop{\kern\z@#1}\limits^{\makebox[0pt][c]{\vbox to-2\ex@{\kern-\tw@\ex@\hbox{\normalfont\scaleddot\kern-0.5pt\scaleddot\kern-0.5pt\scaleddot\kern-0.5pt\scaleddot}\vss}}}}}

\newcommand{\abs}[1]{\lvert#1\rvert}
\newcommand\norm[1]{\left\lVert#1\right\rVert}

\newcommand{\FF}{{\mathcal F}}

\newtheorem{theorem}{Theorem}[section]
\newtheorem{lemma}[theorem]{Lemma}
\newtheorem{proposition}[theorem]{Proposition}

\theoremstyle{definition}

\title{Extended lifespan of the fractional BBM equation}

\author{Dag Nilsson \thanks{The author was supported by an ERCIM `Alain Bensoussan' Fellowship and by grant no. 250070 from the Research Council of Norway}}
\date{}
\begin{document}

\maketitle
\begin{abstract}
For $0<\alpha<1$ and with initial data $\norm{u_0}_{H^{N+\frac{\alpha}{2}}}=\varepsilon$, sufficently small, we show that the existence time for solutions of the fractional BBM equation $\partial_tu+\partial_xu+u\partial_xu+\abs{\mathrm{D}}^\alpha\partial_tu=0$, can be extended beyond the hyperbolic existence time $\frac{1}{\varepsilon}$, to $\frac{1}{\varepsilon^2}$. For the proof we use a modified energy, based on a normal form transformation as in [Hunter, Ifrim, Tataru, Wong, 2015]. In addition we employ ideas and techniques from [Ehrnström, Wang, 2018], in which the authors obtain an enhanced existence time for the fractional KdV equation.
\end{abstract}

\section{Introduction}
We consider the fractional BBM equation
\begin{equation}\label{frac-bbm-eq}
\begin{cases}
&\partial_tu+\partial_xu+u\partial_xu+\abs{\mathrm{D}}^\alpha\partial_tu=0,\\
&u(0,x)=u_0(x),
\end{cases}
\end{equation}
with $\FF(\abs{\mathrm{D}}^\alpha f)(\xi)=\abs{\xi}^\alpha\hat{f}(\xi)$, and where $\mathcal{F}$ is the Fourier transform 
\begin{equation*}
\mathcal{F}(f)(\xi)=\int_\mathbb{R}f(x)\exp(-\mathrm{i}x\xi)\ \mathrm{d}x.
\end{equation*}
Throughout the text we will write $f\lesssim g$, when $\frac{f}{g}$ is uniformly bounded from above, and $f\simeq g$ when $f\lesssim g\lesssim f$.

When $\alpha=2$, \eqref{frac-bbm-eq} is the classical BBM equation introduced in \cite{bbm1972}. In \cite{bonatzvetkov2009} it was shown that the BBM equation is globally well-posed in $H^s(\mathbb{R})$, for $s\geq 0$. This was later generalized in \cite{Bona2009} where the authors showed that \eqref{frac-bbm-eq}, with $1<\alpha\leq 2$, is globally well-posed in $H^s(\mathbb{R})$, for $s\geq 1-\frac{\alpha}{2}$. On the other hand, in \cite{CP2014} it was shown that \eqref{frac-bbm-eq}, with $1<\alpha\leq 2$ is ill-posed in $H^s(\mathbb{R})$ for $s<\max\{0,\frac{3}{2}-\alpha\}$, in the sense that the map $u_0\mapsto u$ is not $C^2$ from $H^s(\mathbb{R})$ to $C([0,T];H^s(\mathbb{R}))$, for any $T>0$. The gap in the theory when $\max\{0,\frac{3}{2}-\alpha\}\leq s<1-\frac{\alpha}{2}$ and $1<\alpha\leq 2$ was filled recently in \cite{Wang2018}, where it was shown that \eqref{frac-bbm-eq} is indeed globally well-posed for such values of $s$. 

For $\alpha=1$, \eqref{frac-bbm-eq} is the regularized Benjamin--Ono equation, which was shown in \cite{bonakalish2000} to be locally well-posed in $H^s(\mathbb{R})$, $s>\frac{1}{2}$, and globally well-posed in $H^s(\mathbb{R})$, $s\geq \frac{3}{2}$.  

When $\alpha\geq 1$, \eqref{frac-bbm-eq} is an ODE in $H^s(\mathbb{R})$, $s>\frac{1}{2}$. This is no longer true when $0<\alpha<1$, which makes this case more difficult. Also, the case $0<\alpha<1$ is interesting from a physical point of view, since $\alpha=\frac{1}{2}$ corresponds to full-dispersion models for gravity waves. To the authors knowledge, the only result on well-posedness of \eqref{frac-bbm-eq} for $0<\alpha<1$ is by Linares, Pilod and Saut \cite{LPS2014}, where they showed using standard energy methods that \eqref{frac-bbm-eq} is locally well-posed in $H^s(\mathbb{R})$, for $s>\frac{3}{2}-\alpha$, obtaining the estimate
\begin{equation}\label{standard-en-est}
\frac{\mathrm{d}}{\mathrm{d}t}\norm{u(t)}_{H^s}^2\lesssim \norm{u(t)}_{H^s}^3,
\end{equation}
which implies that the maximal existence time $T$ for a classical solution of \eqref{frac-bbm-eq}, satisfies $T\gtrsim \frac{1}{\norm{u_0}_{H^s}}$. It is unclear for which values of $\alpha\in(0,1)$ that \eqref{frac-bbm-eq} is expected to be well-posed. Indeed, equation \eqref{frac-bbm-eq} is not invariant under any rescaling $u_\lambda(x,t)=\lambda^au(\lambda^bx,\lambda^ct)$, except for $\alpha=0$, hence is not scaling critical for any $\alpha\in(0,1)$. This is in contrast with the fractional KdV equation which is scaling critical for $\alpha=\frac{1}{2}$, and this value of $\alpha$ is also believed to be critical for the well-posedness theory, as pointed out in \cite{LPS2014}.     

As described in the above paragraph, the well-posedness of \eqref{frac-bbm-eq}, both local and global, is quite well understood for $\alpha\geq 1$, while for $0<\alpha<1$ the question of global well-posedness is completely open. In the present work we therefore consider the question of long time existence, and show that \cite[Theorem 4.12]{LPS2014} can be expanded upon, by extending the lifespan of the solutions.
\begin{theorem}\label{main-res}
Let $0<\alpha<1$ and $N\geq 2$. There exists $\varepsilon_0>0$ such that for any initial data $u_0$ satisfying
\begin{equation*}
\norm{u_0}_{H^{N+\frac{\alpha}{2}}}\leq \varepsilon,
\end{equation*}
with $\varepsilon\leq \varepsilon_0$, there exist $T\gtrsim\frac{1}{\varepsilon^2}$ and a unique solution $u\in C([0,T];H^{N+\frac{\alpha}{2}}(\mathbb{R}))$ of \eqref{frac-bbm-eq} with $u(0,x)=u_0(x)$ such that 
\begin{equation*}
\norm{u}_{C([0,T];H^{N+\frac{\alpha}{2}}(\mathbb{R}))}\lesssim\norm{u_0}_{H^{N+\frac{\alpha}{2}}}.
\end{equation*}
\end{theorem}
For the proof of this theorem we will use the strategy developed by Hunter, Ifrim, Tataru and Wong \cite{HITW2015}, in which a modified energy is defined, based on a normal form transformation, in order to prove enhanced existence time for the Burgers--Hilbert equation. This method was further developed and applied to the full water wave problem in a series of papers \cite{Harrop-Griffiths2017, Hunter2016, Ifrim2017}. However, for our purposes the paper \cite{EW2018} by Ehrnström and Wang is the most relevant one, in which the authors use the method developed in \cite{HITW2015} to obtain an enhanced existence time for the fractional KdV equation. In the present work the symbol of the normal form transformation shares a lot of properties with the corresponding symbol found in \cite{EW2018}, which makes the modified energies for the two equations similar. In particular, we are able to adapt the ideas and techniques developed in \cite{EW2018} to treat the modifed energy.

Theorem \ref{main-res} can be seen as a first step in the investigation of global well-posedness of \eqref{frac-bbm-eq} with $0<\alpha<1$. Indeed, it may be possible to repeat the argument in the proof of Theorem \ref{main-res} and show that the maximal existence time satisfies $T\gtrsim \frac{1}{\epsilon^3}$. In fact, in \cite{Kwon2018}, where the authors study the NLS and modified KdV equations,  this procedure is repeated infinitely many times using an iterative scheme, yielding global well-posedness results for the two equations. However, to the authors knowledge, this method has not been applied to equations involving fractional derivatives, such as \eqref{frac-bbm-eq}. 

Note that in Theorem \ref{main-res}, larger $\alpha$ requires higher regularity of the initial data, which is in contrast with the known well-posedness theory outlined above, where larger $\alpha$ allows for lower regularity of the initial data. This contrast is most likely due to the technique used in the present work, rather than an inherent property of the initial value problem \eqref{frac-bbm-eq}. Indeed, the normal form transformation leads to a natural choice for the modified energy and it is this choice of energy which requires us to have $u\in H^{N+\frac{\alpha}{2}}(\mathbb{R})$.

We point out that the lower bound on $N$ can most likely be decreased to $\frac{3}{2}+\delta$, $\delta>0$, by using that
\begin{equation}\label{embedding}
\norm{u}_{L^\infty}\lesssim \norm{u}_{H^{\frac{1}{2}+\delta}}.
\end{equation}
However, we choose $\delta=\frac{1}{2}$ so that the Sobolev index in \eqref{embedding} is integer valued, making the the proof of Theorem \ref{main-res} less technical. In connection with this we also mention that in \cite{EW2018} the authors require $N\geq 3$. However, this can probably be improved to $N\geq 2$, at least in the case when $0<\alpha<1$. 

In Section \ref{sec-normal-form} we carry out the normal form transformation, with the purpose of removing the quadratic term $u\partial_xu$ in \eqref{frac-bbm-eq}, making the transformed equation cubic. We derive an expression for the symbol $m$ associated with the normal form transformation and establish some growth estimates for it, that are useful when estimating the modified energy.

We proceed in Section \ref{sec-mod-energy} by introducing the modified energy $\mathrm{E}^{(N)}$ as in \cite{HITW2015}, based on the normal form transformation, and show that $\mathrm{E}^{(N)}\simeq \norm{u}_{H^{N+\frac{\alpha}{2}}}^2$, when $\norm{u}_{H^{N+\frac{\alpha}{2}}}$ is sufficiently small.

Section \ref{sec-en-est} is the most technical part of the paper, in which we derive a quartic energy estimate for the modified energy. We show that
\begin{equation}\label{introenest}
\frac{\mathrm{d}}{\mathrm{d}t}\mathrm{E}^{(N)}(t)\lesssim \norm{u(t)}_{H^{N+\frac{\alpha}{2}}}^4,
\end{equation}
which implies that
\begin{equation*}
\mathrm{E}^{(N)}(t)\lesssim \mathrm{E}^{(N)}(0)+\int_0^t\norm{u(s)}_{H^{N+\frac{\alpha}{2}}}^4\ \mathrm{d}s.
\end{equation*}
Using then that $\mathrm{E}^{(N)}\simeq \norm{u}_{H^{N+\frac{\alpha}{2}}}^2$, we obtain
\begin{equation*}
\norm{u(t)}_{H^{N+\frac{\alpha}{2}}}^2\lesssim \norm{u_0}_{H^{N+\frac{\alpha}{2}}}+\int_0^t\norm{u(s)}_{H^{N+\frac{\alpha}{2}}}^4\ \mathrm{d}s.
\end{equation*}
An application of Grönwall's inequality then yields the lower bound for the maximal existence time in Theorem \ref{main-res}. The existence and uniqness part of the theorem then follows as in \cite{LPS2014}.
Below we give a rough outline on how to prove \eqref{introenest}, using methods developed in \cite{EW2018}.

The modified energy is differentiated with respect to time and decomposed into high and low order parts, where the low order parts can be estimated directly, using the growth estimates derived in Section \ref{sec-normal-form}. In the end we are essentially left with two high order terms, $F_{1,0}$, $G_{1,0}$, given by
\begin{align*} 
F_{1,0}&=\int_{\mathbb{R}^3}m(\eta-\sigma,\sigma)(\mathrm{i}\sigma)^k\mathrm{i}(\eta-\sigma)\overline{(\mathrm{i}\xi)^k}\ \mathrm{d}\mathrm{Q}(u),\\
G_{1,0}&=-2\int_{\mathbb{R}^3}\frac{1+\abs{\xi}^\alpha}{1+\abs{\eta}^\alpha}m(\xi-\eta,\eta)\mathrm{i}\eta(\mathrm{i}\sigma)^{k-1}\mathrm{i}(\xi-\eta)\overline{(\mathrm{i}\xi)^k}\ \mathrm{d}\mathrm{Q}(u),
\end{align*}
where
\begin{equation*}
\mathrm{d}\mathrm{Q}(u)=\hat{u}(\eta-\sigma)\hat{u}(\xi-\eta)\hat{u}(\sigma)\overline{\hat{u}(\xi)}.
\end{equation*}
We are able to show that $G_{1,0}\equiv2F_{1,0}$, meaning that $G_{1,0}$ is equal to $2F_{1,0}$, modulo a term that can be estimated by $\norm{u}_{H^{N+\frac{\alpha}{2}}}^2\norm{u}_{H^{k+\frac{\alpha}{2}}}^2$.
We want to estimate $F_{1,0}$ and $G_{1,0}$ with $\norm{u}_{H^{N+\frac{\alpha}{2}}}^2\norm{u}_{H^{k+\frac{\alpha}{2}}}^2$, however, due to the growth estimates on $m$, direct estimates only yield 
\begin{equation*}
\abs{F_{1,0}}+\abs{G_{1,0}}\lesssim\norm{u}_{H^{N+\frac{\alpha}{2}}}^2\norm{u}_{H^{k+1+\frac{\alpha}{2}}}^2.
\end{equation*}
The problem is that there are to many factors $\xi$ and $\sigma$ appearing in $F_{1,0},\ G_{1,0}$. We would like to transfer some of these to the either $\eta-\sigma$ or $\xi-\eta$. Clearly, in some regions of $\mathbb{R}^3$ such a transfer is possible, for instance when $\abs{\xi}\lesssim \abs{\xi-\eta}+\abs{\eta-\sigma}$. The frequency space $\mathbb{R}^3$ can then be decomposed according to whether or not we have this property. This leads to the decomposition $\mathbb{R}^3= \mathcal{A}_1\cup\mathcal{A}_2\cup\mathcal{A}_2^c$, where we in $ \mathcal{A}_1\cup\mathcal{A}_2$ can transfer factors, while in $\mathcal{A}_2^c$ we can not.  Hence, when the domain of integration is restricted to $\mathcal{A}_1\cup\mathcal{A}_2$, $F_{1,0}$, $G_{1,0}$ can be estimated in a straightforward way. The domain $\mathcal{A}_2^c$ still remains, and here we utilize that when integrating over $\mathcal{A}_2^c$, then $G_{1,0}+2F_{1,0}$ is a good commutator, meaning that enough of the factors $\xi$ and $\sigma$ are canceled. This allows us to estimate $\abs{G_{1,0}+2F_{1,0}}\lesssim \norm{u}_{H^{N+\frac{\alpha}{2}}}^2\norm{u}_{H^{k+\frac{\alpha}{2}}}^2$. The terms $F_{1,0}$ and $G_{1,0}$ can then be estimated separately with $\norm{u}_{H^{N+\frac{\alpha}{2}}}^2\norm{u}_{H^{k+\frac{\alpha}{2}}}^2$, by using that $G_{1,0}\equiv 2F_{1,0}$.
\section{Normal form transformation}\label{sec-normal-form}
We introduce a new variable $w$ via a normal form transformation
\begin{equation}\label{normal-form}
w=u+P(u,u),
\end{equation}
where 
\begin{equation}\label{normal-form}
\FF(P(f_1,f_2))(\xi)=\int_\mathbb{R}m(\xi-\eta,\eta)\hat{f}_1(\xi-\eta)\hat{f}_2(\eta)\ \mathrm{d}\eta,
\end{equation}
and where $m$ is to be determined so that $w$ satisfies a cubically nonlinear PDE.
Using \eqref{frac-bbm-eq} we find that
\begin{align*}
&\partial_tw+\partial_xw+\abs{\mathrm{D}}^\alpha\partial_tw\\
&=(1+\abs{\mathrm{D}}^\alpha)\left[\partial_tu+P(\partial_tu,u)+P(u,\partial_tu)\right]+\partial_xu+\partial_xP(u,u)\\
&=(1+\abs{\mathrm{D}}^\alpha)[-(1+\abs{\mathrm{D}}^\alpha)^{-1}(\partial_xu+u\partial_xu)-P((1+\abs{\mathrm{D}}^\alpha)^{-1}(\partial_xu+u\partial_xu),u)\\
&\quad -P(u,(1+\abs{\mathrm{D}}^\alpha)^{-1}(\partial_xu+u\partial_xu))]+\partial_xu+\partial_xP(u,u)\\
&=-u\partial_xu-(1+\abs{\mathrm{D}}^\alpha)P((1+\abs{\mathrm{D}}^\alpha)^{-1}\partial_xu,u)-(1+\abs{\mathrm{D}}^\alpha)P(u,(1+\abs{\mathrm{D}}^\alpha)^{-1}\partial_xu)\\
&\quad +\partial_xP(u,u)\underbrace{-(1+\abs{\mathrm{D}}^\alpha)\left[P((1+\abs{\mathrm{D}}^\alpha)^{-1}(u\partial_xu),u)+P(u,(1+\abs{\mathrm{D}}^\alpha)^{-1}(u\partial_xu))\right]}_{=:R(u)},
\end{align*}
which implies that
\begin{align}
\FF(\partial_tw+\partial_xw+\abs{\mathrm{D}}^\alpha\partial_tw)(\xi)&=\FF(-u\partial_xu-(1+\abs{\mathrm{D}}^\alpha)P((1+\abs{\mathrm{D}}^\alpha)^{-1}\partial_xu,u)\nonumber\\
&\quad -(1+\abs{\mathrm{D}}^\alpha)P(u,(1+\abs{\mathrm{D}}^\alpha)^{-1}\partial_xu)\nonumber\\
&\quad +\partial_xP(u,u)+R(u))(\xi).\label{linear-eq}
\end{align}
The function $m$ is chosen in such a way as to remove the quadratic terms in \eqref{linear-eq}, that is, $m$ must satisfy
\begin{align*}
&\int_\mathbb{R}\bigg[-\frac{\xi}{2}+m(\xi-\eta,\eta)\bigg(\xi-(1+\abs{\xi}^\alpha)(1+\abs{\xi-\eta}^\alpha)^{-1}(\xi-\eta)\\
&\quad -(1+\abs{\xi}^\alpha)(1+\abs{\eta}^\alpha)^{-1}\eta\bigg)\bigg]\hat{u}(\xi-\eta)\hat{u}(\eta)\ \mathrm{d}\eta=0
\end{align*}
which holds if and only if
\begin{equation}\label{formula-m}
m(\xi-\eta,\eta)=\frac{\xi(1+\abs{\xi-\eta}^\alpha)(1+\abs{\eta}^\alpha)}{2[\xi(1+\abs{\eta}^\alpha)(\abs{\xi-\eta}^\alpha-\abs{\xi}^\alpha)-\eta(1+\abs{\xi}^\alpha)(\abs{\xi-\eta}^\alpha-\abs{\eta}^\alpha)]}.
\end{equation}
Note that $m$ is symmetric in $(\xi-\eta)$ and $\eta$, that is $m(\xi-\eta,\eta)=m(\eta,\xi-\eta)$. In addition, $m$ satisifies
\begin{equation}\label{functional-rel}
m(\xi-\eta,\eta)\eta(1+\abs{\xi}^\alpha)+m(\eta-\xi,\xi)\xi(1+\abs{\eta}^\alpha)=0.
\end{equation}
Similar to \cite[Proposition 2.1]{EW2018}, we have the following result.
\begin{proposition}\label{prop-m}
The symbol $m$ satisfies 
\begin{align}
\frac{\abs{\xi-\eta}}{\abs{\eta}}+\frac{\abs{\eta}}{\abs{\xi-\eta}}\lesssim\abs{m(\xi-\eta,\eta)}&\lesssim\frac{\abs{\xi-\eta}^{1-\alpha}}{\abs{\eta}}+\frac{\abs{\eta}^{1-\alpha}}{\abs{\xi-\eta}},\quad \text{for } (\xi-\eta)^2+\eta^2\leq 1,\label{est-1}\\
\frac{\abs{\xi-\eta}^{1-\alpha}}{\abs{\eta}}+\frac{\abs{\eta}^{1-\alpha}}{\abs{\xi-\eta}}\lesssim\abs{m(\xi-\eta,\eta)}&\lesssim \frac{\abs{\xi-\eta}}{\abs{\eta}}+\frac{\abs{\eta}}{\abs{\xi-\eta}},\quad \text{for } (\xi-\eta)^2+\eta^2\geq 1\label{est-2}.
\end{align}
\end{proposition}
\begin{proof}
We introduce polar coordinates
\begin{equation*}
r\cos(\theta)=\xi-\eta,\quad r\sin(\theta)=\eta,\quad r\geq 0,\ 0\leq \theta<2\pi,
\end{equation*}
so that 
\begin{equation*}
m(\xi-\eta,\eta)=\frac{r(\cos(\theta)+\sin(\theta))(1+r^\alpha\abs{\cos(\theta)}^\alpha)(1+r^\alpha\abs{\sin(\theta)}^\alpha)}{2n(r\cos(\theta),r\sin(\theta))},
\end{equation*}
where
\begin{align*}
n(r\cos(\theta),r\sin(\theta))&=r^{1+\alpha}\big[(\cos(\theta)+\sin(\theta))(1+r^\alpha\abs{\sin(\theta)}^\alpha)(\abs{\cos(\theta)}^\alpha-\abs{\cos(\theta)+\sin(\theta)}^\alpha)\\
&\quad -\sin(\theta)(1+r^\alpha\abs{\cos(\theta)}^\alpha)(\abs{\cos(\theta)}^\alpha-\abs{\sin(\theta)}^\alpha)\big]\\
&=:r^{1+\alpha}\tilde{n}(r\cos(\theta),r\sin(\theta)).
\end{align*}
We have that $\tilde{n}(r\cos(\theta),r\sin(\theta))=0$ if and only if either $\cos(\theta)=0$, $\sin(\theta)=0$ or $\cos(\theta)+\sin(\theta)=0$. Moreover, these zeros are all of order $1$. It follows that
\begin{equation*}
n(r\cos(\theta),r\sin(\theta)=r^{1+\alpha}\cos(\theta)\sin(\theta)(\cos(\theta)+\sin(\theta))h(r,\theta),
\end{equation*}
where $h$ is a function that is bounded away from $0$ and $\abs{h(r,\theta)}\simeq r^\alpha$ for $r\geq 1$ and $h$ is bounded for $r\leq 1$. Hence, 
\begin{equation*}
\abs{m(\xi-\eta,\eta)}=\frac{r^{2-\alpha}(1+\abs{\xi-\eta}^\alpha)(1+\abs{\eta}^\alpha)}{2\abs{(\xi-\eta)\eta h(r,\theta)}}\lesssim\begin{cases}
&\frac{\abs{\xi-\eta}^{1-\alpha}}{\abs{\eta}}+\frac{\abs{\eta}^{1-\alpha}}{\abs{\xi-\eta}},\quad (\xi-\eta)^2+\eta^2\leq 1,\\
&\frac{\abs{\xi-\eta}}{\abs{\eta}}+\frac{\abs{\eta}}{\abs{\xi-\eta}},\quad (\xi-\eta)^2+\eta^2\geq 1.
\end{cases}
\end{equation*}
and
\begin{equation*}
\abs{m(\xi-\eta,\eta)}=\frac{r^{2-\alpha}(1+\abs{\xi-\eta}^\alpha)(1+\abs{\eta}^\alpha)}{2\abs{(\xi-\eta)\eta h(r,\theta)}}\gtrsim\begin{cases}
&\frac{\abs{\xi-\eta}}{\abs{\eta}}+\frac{\abs{\eta}}{\abs{\xi-\eta}},\quad (\xi-\eta)^2+\eta^2\leq 1,\\
&\frac{\abs{\xi-\eta}^{1-\alpha}}{\abs{\eta}}+\frac{\abs{\eta}^{1-\alpha}}{\abs{\xi-\eta}},\quad (\xi-\eta)^2+\eta^2\geq 1.
\end{cases}
\end{equation*}
\end{proof}
\section{The modified energy}\label{sec-mod-energy}
In the previous section we introduced in \eqref{normal-form} a new variable $w$ via a normal form transformation, with $w$ satisfying the PDE
\begin{equation}\label{PDE-w}
\partial_tw+\partial_xw+\abs{\mathrm{D}}^\alpha\partial_tw=R(u),
\end{equation}
where $R(u)$ is cubic in $u$. There is a loss of derivatives when applying the standard energy method directly to \eqref{PDE-w}. Because of this we follow \cite{HITW2015}, and continue to work with \eqref{frac-bbm-eq}, but introduce a suitable modified energy. In order to find such an energy, we first use \eqref{normal-form}, and note that
\begin{align}
\norm{\partial_x^k(1+\abs{\mathrm{D}}^\alpha)^\frac{1}{2}w}_{L^2(\mathbb{R})}^2&=\norm{\partial_x^k(1+\abs{\mathrm{D}}^\alpha)^\frac{1}{2}u}_{L^2(\mathbb{R})}^2+2\langle \partial_x^k (1+\abs{\mathrm{D}}^\alpha)^\frac{1}{2}u,\partial_x^k(1+\abs{\mathrm{D}}^\alpha)^\frac{1}{2}P(u,u)\rangle\nonumber\\
&\quad +\norm{\partial_x^k(1+\abs{\mathrm{D}}^\alpha)^\frac{1}{2}P(u,u)}_{L^2(\mathbb{R})}^2,\label{mod-energy-motivation}
\end{align}
where $\langle f,g\rangle=\int_\mathbb{R}f\bar{g}\ \mathrm{d}x$. Using \eqref{mod-energy-motivation} as motivation, we define the kth partial energy
\begin{equation}\label{modified-energy}
\mathrm{E}_k(t):=\norm{\partial_x^k(1+\abs{\mathrm{D}}^\alpha)^\frac{1}{2}u}_{L^2(\mathbb{R})}^2+2\langle \partial_x^k(1+\abs{\mathrm{D}}^\alpha)^\frac{1}{2}u,\partial_x^k(1+\abs{\mathrm{D}}^\alpha)^\frac{1}{2}P(u,u)\rangle.
\end{equation}
We disregard the term $\norm{\partial_x^k(1+\abs{\mathrm{D}}^\alpha)^\frac{1}{2}P(u,u)}_{L^2(\mathbb{R})}^2$ in \eqref{mod-energy-motivation}, since this is not comparable to the $H^{k+\frac{\alpha}{2}}$-norm. 
Note that $\mathrm{E}_k(t)$ contains the fractional derivative $(1+\mathrm{D}^\alpha)^\frac{1}{2}$, and this is to accommodate for the terms $(1+\abs{\xi}^\alpha),\ (1+\abs{\eta}^\alpha)$ appearing in the functional relation \eqref{functional-rel}. This is in contrast with \cite[equation (2.9)]{EW2018}, where there are no such terms appearing in the functional relation.

We are now ready to define modified energy, and to show that it is equivalent to the $H^{N+\frac{\alpha}{2}}(\mathbb{R})$-norm.
\begin{lemma}
\begin{equation}\label{equivalence}
\mathrm{E}^{(N)}(t):=\sum_{k=1}^N\mathrm{E}_k(t)+\norm{(1+\abs{\mathrm{D}}^\alpha)^\frac{1}{2}u}_{L^2(\mathbb{R})}^2\simeq \norm{u}_{H^{N+\frac{\alpha}{2}}(\mathbb{R})}^2,
\end{equation}
uniformly for $\norm{u}_{H^{N+\frac{\alpha}{2}}(\mathbb{R})}<\varepsilon$.
\end{lemma}
\begin{proof}
In order to establish \eqref{equivalence} it is sufficient to show that
\begin{equation}\label{suff-est}
\abs{\langle \partial_x^k(1+\abs{\mathrm{D}}^\alpha)^\frac{1}{2}u,\partial_x^k(1+\abs{\mathrm{D}}^\alpha)^\frac{1}{2}P(u,u)\rangle}\lesssim \epsilon\norm{u}_{H^{k+\frac{\alpha}{2}}(\mathbb{R})}^2.
\end{equation} 
A first step towards achieving this is to decompose
\begin{align*}
\langle \partial_x^k(1+\abs{\mathrm{D}}^\alpha)^\frac{1}{2}u,\partial_x^k(1+\abs{\mathrm{D}}^\alpha)^\frac{1}{2}P(u,u)\rangle&=2\langle (1+\abs{\mathrm{D}}^\alpha)^\frac{1}{2}\partial_x^ku,(1+\abs{\mathrm{D}}^\alpha)^\frac{1}{2}P(u,\partial_x^ku)\rangle\\
&\quad +\sum_{j=1}^{k-1}c_{k,j}\langle (1+\abs{\mathrm{D}}^\alpha)^\frac{1}{2}\partial_x^ku,(1+\abs{\mathrm{D}}^\alpha)^\frac{1}{2}P(\partial_x^ju,\partial_x^{k-j}u)\rangle\\
&=:2A_0+\sum_{j=1}^{k-1}c_{k,j}A_j,
\end{align*}
where $c_{k,j}$ are the binomial coefficients. Due to the properties of $m$ described in Proposition \eqref{prop-m}, $A_0$ is the worst term to estimate, and we will treat it using change of variables, integration by parts and \eqref{functional-rel}.
\begin{align}
A_0&=\int_{\mathbb{R}^2}m(\xi-\eta,\eta)\hat{u}(\xi-\eta)(\mathrm{i}\eta)^k\hat{u}(\eta)(1+\abs{\xi}^\alpha)\overline{(\mathrm{i}\xi)^k\hat{u}(\xi)}\ \mathrm{d}\eta\ \mathrm{d}\xi\nonumber\\
&=-\int_{\mathbb{R}^2}m(\xi-\eta,\eta)\mathrm{i}(\xi-\eta)\hat{u}(\xi-\eta)(\mathrm{i}\eta)^k\hat{u}(\eta)(1+\abs{\xi}^\alpha)\overline{(\mathrm{i}\xi)^{k-1}\hat{u}(\xi)}\ \mathrm{d}\eta\ \mathrm{d}\xi\nonumber\\
&\quad -\int_{\mathbb{R}^2}m(\xi-\eta,\eta)\hat{u}(\xi-\eta)(\mathrm{i}\eta)^{k+1}\hat{u}(\eta)(1+\abs{\xi}^\alpha)\overline{(\mathrm{i}\xi)^{k-1}\hat{u}(\xi)} \mathrm{d}\eta\ \mathrm{d}\xi\nonumber\\
&=:A_0^1+A_0^2,\label{A_0-calc}
\end{align}
and
\begin{align*}
A_0^2&=-\int_{\mathbb{R}^2}m(\xi-\eta,\eta)\hat{u}(\xi-\eta)(\mathrm{i}\eta)^{k+1}\hat{u}(\eta)(1+\abs{\xi}^\alpha)\overline{(\mathrm{i}\xi)^{k-1}\hat{u}(\xi)}\ \mathrm{d}\eta\ \mathrm{d}\xi\\
&=-\int_{\mathbb{R}^2}m(\xi-\eta,\eta)\hat{u}(\eta-\xi)(\mathrm{i}\xi)^{k-1}\hat{u}(\xi)(1+\abs{\xi}^\alpha)\overline{(\mathrm{i}\eta)^{k+1}\hat{u}(\eta)}\ \mathrm{d}\eta\ \mathrm{d}\xi\\
&=-\int_{\mathbb{R}^2}m(\eta-\xi,\xi)\hat{u}(\xi-\eta)(\mathrm{i}\eta)^{k-1}\hat{u}(\eta)(1+\abs{\eta}^\alpha)\overline{(\mathrm{i}\xi)^{k+1}\hat{u}(\xi)}\ \mathrm{d}\eta\ \mathrm{d}\xi\\
&=-\int_{\mathbb{R}^2}m(\xi-\eta,\eta)\hat{u}(\xi-\eta)(\mathrm{i}\eta)^k\hat{u}(\eta)(1+\abs{\xi}^\alpha)\overline{(\mathrm{i}\xi)^k\hat{u}(\xi)}\ \mathrm{d}\eta\ \mathrm{d}\xi\\
&=-A_0.
\end{align*}
where we in the second equality made the change of varibles $(\xi,\eta)\leftrightarrow -(\xi,\eta)$, in the third equality we made the change of variables $(\xi,\eta)\leftrightarrow (\eta,\xi)$ and in the fourth we used \eqref{functional-rel}.
Hence, it follows from \eqref{A_0-calc} that $2A_0=A_0^1$. It remains to estimate $A_0^1$. From Proposition \ref{prop-m} we  know that $m$ has singularities at $\eta=0$ and $\xi-\eta=0$. However, in $A_0^1$ there is a factor $(\xi-\eta)(\mathrm{i}\eta)^k$ appearing which cancels out these singularities. It is therefore enough to estimate the high frequencies. Using \eqref{est-2}, we find that
\begin{equation}\label{est-xi}
\abs{m(\xi-\eta,\eta)}\lesssim \frac{\abs{\xi-\eta}}{\abs{\eta}}+1+\frac{\abs{\xi}}{\abs{\xi-\eta}},\quad \text{for } (\xi-\eta)^2+\eta^2\geq 1.
\end{equation}
Equation \eqref{est-xi} can then be used to estimate the high frequency part of $A_0^1$.  
\begin{align*}
&\abs{\int_{(\xi-\eta)^2+\eta^2\geq 1}m(\xi-\eta,\eta)\mathrm{i}(\xi-\eta)\hat{u}(\xi-\eta)(\mathrm{i}\eta)^k\hat{u}(\eta)(1+\abs{\xi}^\alpha)\overline{(\mathrm{i}\xi)^{k-1}\hat{u}(\xi)}\ \mathrm{d}\eta\ \mathrm{d}\xi}\\
&\quad \lesssim\int_{\mathbb{R}^2}\abs{(\xi-\eta)^2\hat{u}(\xi-\eta)\eta^{k-1}\hat{u}(\eta)(1+\abs{\xi}^\alpha)\xi^{k-1}\hat{u}(\xi)}\ \mathrm{d}\eta\ \mathrm{d}\xi\\
&\qquad +\int_{\mathbb{R}^2}\abs{(\xi-\eta)\hat{u}(\xi-\eta)\eta^k\hat{u}(\eta)(1+\abs{\xi}^\alpha)\xi^{k-1}\hat{u}(\xi)}\ \mathrm{d}\eta\ \mathrm{d}\xi\\
&\qquad +\int_{\mathbb{R}^2}\abs{\hat{u}(\xi-\eta)\eta^k\hat{u}(\eta)(1+\abs{\xi}^\alpha)\xi^k\hat{u}(\xi)}\ \mathrm{d}\eta\ \mathrm{d}\xi\\
&\quad =:A_0^{1,1}+A_0^{1,2}+A_0^{1,3}.
\end{align*}
We proceed by estimating $A_0^{1,1},A_0^{1,2}$ and $A_0^{1,3}$ directly:
\begin{align*}
A_0^{1,1}&=\int_{\mathbb{R}^2}\abs{(\xi-\eta)^2\hat{u}(\xi-\eta)\eta^{k-1}\hat{u}(\eta)(1+\abs{\xi}^\alpha)\xi^{k-1}\hat{u}(\xi)}\ \mathrm{d}\eta\ \mathrm{d}\xi\\
&\leq\norm{\partial_x^2u\partial_x^{k-1}u}_{L^2(\mathbb{R})}\norm{(1+\abs{\mathrm{D}}^\alpha)\partial_x^{k-1}u}_{L^2(\mathbb{R})}\\
&\leq \norm{u}_{H^2}\norm{\partial_x^{k-1}u}_{L^\infty}\norm{u}_{H^{k-1+\alpha}}\\
&\lesssim \norm{u}_{H^2(\mathbb{R})}\norm{u}_{H^{k}(\mathbb{R})}^2,
\end{align*}
\begin{align*}
A_0^{1,2}&=\int_{\mathbb{R}^2}\abs{(\xi-\eta)\hat{u}(\xi-\eta)\eta^k\hat{u}(\eta)(1+\abs{\xi}^\alpha)\xi^{k-1}\hat{u}(\xi)}\ \mathrm{d}\eta\ \mathrm{d}\xi\\
&\leq \norm{\partial_xu(1+\abs{\mathrm{D}}^\alpha)\partial_x^{k-1}u}_{L^2(\mathbb{R})}\norm{\partial_x^ku}_{L^2(\mathbb{R})}\\
&\leq \norm{\partial_xu}_{L^\infty}\norm{u}_{H^{k-1+\alpha}}\norm{u}_{H^k}\\
 &\lesssim \norm{u}_{H^2(\mathbb{R})}\norm{u}_{H^k(\mathbb{R})}^2,
 \end{align*}
\begin{align*}
A_0^{1,3}&=\int_{\mathbb{R}^2}\abs{\hat{u}(\xi-\eta)\eta^k\hat{u}(\eta)(1+\abs{\xi}^\alpha)\xi^k\hat{u}(\xi)}\ \mathrm{d}\eta\ \mathrm{d}\xi\\
&=\int_{\mathbb{R}^2}\abs{\hat{u}(\xi-\eta)\eta^k\hat{u}(\eta)(1+\abs{\xi-\eta+\eta}^\frac{\alpha}{2}\abs{\xi}^\frac{\alpha}{2})\xi^k\hat{u}(\xi)}\ \mathrm{d}\eta\ \mathrm{d}\xi\\
&\lesssim \int_{\mathbb{R}^2}\abs{\hat{u}(\xi-\eta)\eta^k\hat{u}(\eta)(1+\abs{\xi-\eta}^\frac{\alpha}{2})(1+\abs{\xi}^\frac{\alpha}{2})\xi^k\hat{u}(\xi)}\ \mathrm{d}\eta\ \mathrm{d}\xi\\
&\quad +\int_{\mathbb{R}^2}\abs{\hat{u}(\xi-\eta)\eta^k\hat{u}(\eta)(1+\abs{\eta}^\frac{\alpha}{2})(1+\abs{\xi}^\frac{\alpha}{2})\xi^k\hat{u}(\xi)}\ \mathrm{d}\eta\ \mathrm{d}\xi\\
&\lesssim \norm{u}_{H^{1+\frac{\alpha}{2}}(\mathbb{R})}\norm{u}_{H^{k+\frac{\alpha}{2}}(\mathbb{R})}\norm{u}_{H^k(\mathbb{R})}+\norm{u}_{H^1(\mathbb{R})}\norm{u}_{H^{k+\frac{\alpha}{2}}(\mathbb{R})}^2.
\end{align*}
Combining the above estimates gives us the desired estimate:
\begin{equation*}
\abs{A_0^1}\lesssim \varepsilon\norm{u}_{H^{k+\frac{\alpha}{2}}(\mathbb{R})}^2.
\end{equation*}
The terms $A_j$ can be estimated directly by $\varepsilon\norm{u}_{H^{k+\frac{\alpha}{2}}(\mathbb{R})}^2$ using \eqref{est-2} and we therefore omit the details.
\end{proof}
\section{The energy estimates}\label{sec-en-est}
This section is devoted to the proof of the following energy inequality.
\begin{proposition}\label{en-estimates}
For $k\geq 1$,
\begin{equation}\label{kmain-en-est}
\frac{\mathrm{d}}{\mathrm{d}t}\mathrm{E}_k(t)\lesssim \norm{u(t)}_{H^{2+\frac{\alpha}{2}}}^2\norm{u(t)}_{H^{k+\frac{\alpha}{2}}}^2+\norm{u(t)}_{H^{k}}^4.
\end{equation}
\end{proposition}
The energy inequality \eqref{introenest} then follows by summing over $k$ in \eqref{kmain-en-est} and using the fact that $\norm{(1+\abs{\mathrm{D}}^\alpha)^\frac{1}{2}u}_{L^2(\mathbb{R})}^2$ is conserved by solutions of \eqref{frac-bbm-eq}.

We first note that
\begin{align*}
\frac{1}{2}\frac{\mathrm{d}}{\mathrm{d}t}\mathrm{E}_k(t)&=\langle (1+\abs{\mathrm{D}}^\alpha)^\frac{1}{2}\partial_x^k\partial_tu,(1+\abs{\mathrm{D}}^\alpha)^\frac{1}{2}\partial_x^k u\rangle+\langle \partial_x^k\partial_tu,(1+\abs{\mathrm{D}}^\alpha)^\frac{1}{2}\partial_x^kP(u,u)\rangle\\
&\quad+ 2\langle\partial_x^ku,(1+\abs{\mathrm{D}}^\alpha)^\frac{1}{2}\partial_x^kP(\partial_tu,u)\rangle\\
&=\langle(1+\abs{\mathrm{D}}^\alpha)\partial_x^k(-(1+\abs{\mathrm{D}}^\alpha)^{-1}(\partial_xu+u\partial_xu)), \partial_x^ku\rangle\\
&\quad+\langle\partial_x^k(-(1+\abs{\mathrm{D}}^\alpha)^{-1}(\partial_xu+u\partial_xu)),(1+\abs{\mathrm{D}}^\alpha)\partial_x^kP(u,u)\rangle\\
&\quad +2\langle \partial_x^ku,(1+\abs{\mathrm{D}}^\alpha)\partial_x^kP(-(1+\abs{\mathrm{D}}^\alpha)^{-1}(\partial_xu+u\partial_xu),u)\rangle\\
&=\langle \partial_x^k(-u\partial_xu+\partial_xP(u,u)-2(1+\abs{\mathrm{D}}^\alpha)P((1+\abs{\mathrm{D}}^\alpha)^{-1}\partial_xu,u)),\partial_x^ku\rangle\\
&\quad -\langle \partial_x^k(u\partial_xu),\partial_x^kP(u,u)\rangle-2\langle \partial_x^ku,(1+\abs{\mathrm{D}}^\alpha)\partial_x^kP((1+\abs{\mathrm{D}}^\alpha)^{-1}(u\partial_xu),u)\rangle\\
&=-\langle \partial_x^k(u\partial_xu),\partial_x^kP(u,u)\rangle-2\langle \partial_x^ku,(1+\abs{\mathrm{D}}^\alpha)\partial_x^kP((1+\abs{\mathrm{D}}^\alpha)^{-1}(u\partial_xu),u)\rangle, 
\end{align*}
where we in the last equality used the definition of $m$. We decompose further by writing
\begin{align*}
-\langle \partial_x^k(u\partial_xu),\partial_x^kP(u,u)\rangle&=-\frac{1}{2}\langle\partial_x^{k+1}(u^2),\partial_x^kP(u,u)\rangle\\
&=\frac{1}{2}\langle \partial_x^k(u^2),\partial_x^{k+1}P(u,u)\rangle\\
&=\underbrace{\langle\partial_x^k(u^2),P(\partial_x^{k+1}u,u)\rangle}_{=:F_0}+\sum_{j=1}^kc_{k,j}\underbrace{\frac{1}{2}\langle\partial_x^k(u^2),P(\partial_x^{k+1-j}u,\partial_x^ju)\rangle}_{=:F_j},
\end{align*}
and
\begin{align*}
&-2\langle \partial_x^ku,(1+\abs{\mathrm{D}}^\alpha)\partial_x^kP((1+\abs{\mathrm{D}}^\alpha)^{-1}(u\partial_xu),u)\rangle\\
&\quad=-\langle\partial_x^ku,(1+\abs{\mathrm{D}}^\alpha)\partial_x^kP((1+\abs{\mathrm{D}}^\alpha)^{-1}\partial_x(u^2),u)\rangle\\
&\quad=\underbrace{-\langle \partial_x^ku,(1+\abs{\mathrm{D}}^\alpha)P((1+\abs{\mathrm{D}}^\alpha)^{-1}\partial_x^{k+1}(u^2),u)\rangle}_{=:G_0}\\
&\qquad +\sum_{j=1}^kc_{k,j}\left(\underbrace{-\langle \partial_x^ku,(1+\abs{\mathrm{D}}^\alpha)P((1+\abs{\mathrm{D}}^\alpha)^{-1}\partial_x^{k+1-j}(u^2),\partial_x^ju)\rangle}_{=:G_j} \right).
\end{align*}
Hence, we have that
\begin{equation}\label{energy-decomposition}
\frac{1}{2}\frac{\mathrm{d}}{\mathrm{d}t}\mathrm{E}_k(t)=F_0+G_0+\sum_{j=1}^kc_{k,j}\left(F_j+G_j\right).
\end{equation}
The task is now to estimate each term in \eqref{energy-decomposition}, and we start by considering the worst terms $F_0$ and $G_0$.
\begin{lemma}\label{energy-first-decomp}
For $k\geq 1$,
\begin{equation*}
\frac{1}{2}\frac{\mathrm{d}}{\mathrm{d}t}\mathrm{E}_k(t)= \sum_{j=1}^kc_{k,j}\left(F_j+G_j\right).
\end{equation*}
\end{lemma}
\begin{proof}
Using the same methods as in the proof of Lemma \ref{equivalence}, we find that
\begin{align*}
G_0&=-\langle \partial_x^ku,(1+\abs{\mathrm{D}}^\alpha)P((1+\abs{\mathrm{D}}^\alpha)^{-1}\partial_x^{k+1}(u^2),u)\rangle\\
&=-\int_{\mathbb{R}^2}(1+\abs{\xi}^\alpha)m(\xi-\eta,\eta)(1+\abs{\eta}^\alpha)^{-1}(\mathrm{i}\eta)^{k+1}\FF(u^2)(\eta)\hat{u}(\xi-\eta)\overline{(\mathrm{i}\xi)^k\hat{u}(\xi)}\ \mathrm{d}\eta\ \mathrm{d}\xi\\
&=-\int_{\mathbb{R}^2}(1+\abs{\eta}^\alpha)m(\eta-\xi,\xi)(1+\abs{\xi}^\alpha)^{-1}(\mathrm{i}\xi)^{k+1}\FF(u^2)(\xi)\hat{u}(\eta-\xi)\overline{(\mathrm{i}\eta)^k\hat{u}(\eta)}\ \mathrm{d}\eta\ \mathrm{d}\xi\\
&=-\int_{\mathbb{R}^2}m(\xi-\eta,\eta)(\mathrm{i}\xi)^k\FF(u^2)(\xi)\hat{u}(\eta-\xi)\overline{(\mathrm{i}\eta)^{k+1}\hat{u}(\eta)}\ \mathrm{d}\eta\ \mathrm{d}\xi\\
&=-\int_{\mathbb{R}^2}m(\xi-\eta,\eta)(\mathrm{i}\eta)^{k+1}\hat{u}(\eta)\hat{u}(\xi-\eta)\overline{(\mathrm{i}\xi)^k\FF(u^2)(\xi)}\ \mathrm{d}\eta\ \mathrm{d}\xi\\
&=-F_0
\end{align*}
\end{proof}
Before continuing to estimate the remaining terms, we note that $F_1=F_k$ and we can also relate $F_1$ and $G_1$:
\begin{lemma}\label{F_1-G_1}
For $k\geq 1$,
\begin{equation*}
G_1=2F_1+\mathcal{O}(\norm{u}_{H^2(\mathbb{R})}\norm{u}_{H^k(\mathbb{R})}^3).
\end{equation*}
\end{lemma}
\begin{proof}
Using change of variables and \eqref{functional-rel}, we find that
\begin{align}
G_1&=-\langle \partial_x^ku,(1+\abs{\mathrm{D}}^\alpha)P((1+\abs{\mathrm{D}}^\alpha)^{-1}\partial_x^k(u^2),\partial_xu)\rangle\nonumber\\
&=-\int_{\mathbb{R}^2}(1+\abs{\xi}^\alpha)m(\xi-\eta,\eta)(1+\abs{\eta}^\alpha)^{-1}(\mathrm{i}\eta)^k\FF(u^2)(\eta)\mathrm{i}(\xi-\eta)\hat{u}(\xi-\eta)\overline{(\mathrm{i}\xi)^k\hat{u}(\xi)}\ \mathrm{d}\eta\ \mathrm{d}\xi\nonumber\\
&=-\int_{\mathbb{R}^2}(1+\abs{\eta}^\alpha)m(\eta-\xi,\xi)(1+\abs{\xi}^\alpha)^{-1}(\mathrm{i}\xi)^k\FF(u^2)(\xi)\mathrm{i}(\eta-\xi)\hat{u}(\eta-\xi)\overline{(\mathrm{i}\eta)^k\hat{u}(\eta)}\ \mathrm{d}\eta\ \mathrm{d}\xi\nonumber\\
&=-\int_{\mathbb{R}^2}m(\xi-\eta,\eta)(\mathrm{i}\xi)^{k-1}\FF(u^2)(\xi)\mathrm{i}(\eta-\xi)\hat{u}(\eta-\xi)\overline{(\mathrm{i}\eta)^{k+1}\hat{u}(\eta)}\ \mathrm{d}\eta\ \mathrm{d}\xi\nonumber\\
&=-\int_{\mathbb{R}^2}m(\xi-\eta,\eta)(\mathrm{i}\eta)^{k+1}\hat{u}(\eta)\mathrm{i}(\xi-\eta)\hat{u}(\xi-\eta)\overline{(\mathrm{i}\xi)^{k-1}\FF(u^2)(\xi)}\ \mathrm{d}\eta\ \mathrm{d}\xi\nonumber\\
&=-\langle P(\partial_x^{k+1}u,\partial_xu),\partial_x^{k-1}(u^2)\rangle,\label{id-g1}
\end{align}
and 
\begin{align}
2F_1&=-\langle \partial_xP(\partial_x^{k}u,\partial_xu),\partial_x^{k-1}(u^2)\rangle\nonumber\\
&=-\langle P(\partial_x^{k+1}u,\partial_xu),\partial_x^{k-1}(u^2)\rangle-\langle P(\partial_x^ku,\partial_x^2u),\partial_x^{k-1}(u^2)\rangle\nonumber\\
&=G_1-\langle P(\partial_x^ku,\partial_x^2u),\partial_x^{k-1}(u^2)\rangle,\label{id-f1}
\end{align}
where we in the last step used \eqref{id-g1}. Next we estimate $\langle P(\partial_x^ku,\partial_x^2u),\partial_x^{k-1}(u^2)\rangle$ using \eqref{est-xi}. Just as in the proof of Lemma \ref{equivalence} it is enough to estimate the high frequencies. 
\begin{align*}
&\abs{\int_{(\xi-\eta)^2+\eta^2\geq 1}m(\xi-\eta,\eta)(\mathrm{i}\eta)^k\hat{u}(\eta)(\mathrm{i}(\xi-\eta))^2\hat{u}(\xi-\eta)\overline{(\mathrm{i}\xi)^{k-1}\FF(u^2)(\xi)}\ \mathrm{d}\eta\ \mathrm{d}\xi}\\
&\quad \int_\mathbb{R}\abs{\eta^{k-1}\hat{u}(\eta)(\xi-\eta)^3\hat{u}(\xi-\eta)\xi^{k-1}\mathcal{F}(u^2)(\xi)}\ \mathrm{d}\eta\ \mathrm{d}\xi\\
&\quad +\int_\mathbb{R}\abs{\eta^{k}\hat{u}(\eta)(\xi-\eta)^2\hat{u}(\xi-\eta)\xi^{k-1}\mathcal{F}(u^2)(\xi)}\ \mathrm{d}\eta\ \mathrm{d}\xi\\
&\quad \int_\mathbb{R}\abs{\eta^{k}\hat{u}(\eta)(\xi-\eta)\hat{u}(\xi-\eta)\xi^{k}\mathcal{F}(u^2)(\xi)}\ \mathrm{d}\eta\ \mathrm{d}\xi\\
&=:B_1+B_2+B_3
\end{align*}
For the term $B_1$ we first use the triangle inequality to get that
\begin{align*}
B_1&\leq \int_\mathbb{R}\abs{\eta^{k}\hat{u}(\eta)(\xi-\eta)^2\hat{u}(\xi-\eta)\xi^{k-1}\mathcal{F}(u^2)(\xi)}\ \mathrm{d}\eta\ \mathrm{d}\xi+\int_\mathbb{R}\abs{\eta^{k-1}\hat{u}(\eta)(\xi-\eta)^2\hat{u}(\xi-\eta)\xi^{k}\mathcal{F}(u^2)(\xi)}\ \mathrm{d}\eta\ \mathrm{d}\xi\\
&=B_2+\int_\mathbb{R}\abs{\eta^{k-1}\hat{u}(\eta)(\xi-\eta)^2\hat{u}(\xi-\eta)\xi^{k}\mathcal{F}(u^2)(\xi)}\ \mathrm{d}\eta\ \mathrm{d}\xi,
\end{align*}
and
\begin{align*}
\int_\mathbb{R}\abs{\eta^{k-1}\hat{u}(\eta)(\xi-\eta)^2\hat{u}(\xi-\eta)\xi^{k}\mathcal{F}(u^2)(\xi)}\ \mathrm{d}\eta\ \mathrm{d}\xi&\leq \norm{\partial_x^{k-1}u\partial_x^2u}_{L^2}\norm{\partial_x^k(u^2)}_{L^2}\\
&\leq \norm{\partial_x^{k-1}u}_{L^\infty}\norm{u}_{H^2}\norm{u}_{H^k}^2\\
&\lesssim \norm{u}_{H^2}\norm{u}_{H^k}^3.
\end{align*}
In the same way we have that
\begin{align*}
B_2&\leq \norm{\partial_x^k}_{L^2}\norm{\partial_x^2\partial_x^{k-1}(u^2)}_{L^2}\lesssim \norm{u}_{H^2}\norm{u}_{H^k}^3,\\
B_3&\leq \norm{\partial_x^ku}_{L^2}\norm{\partial_xu\partial_x^k(u^2)}_{L^2}\lesssim\norm{u}_{H^2}\norm{u}_{H^k}^3.
\end{align*}
Hence,
\begin{equation*}
B_1+B_2+B_3\lesssim \norm{u}_{H^2}\norm{u}_{H^k}^3.
\end{equation*}
\end{proof}
The task is now to estimate $F_j,$ $G_j$ for $j=1,2,\ldots, k$. The problematic terms are $F_1$, $G_1$, $F_k$ and $G_k$, while $F_j$, $G_j$, $j\in \{2,3\ldots, k-1\}$ can be estimated directly. However, we first consider the special case when $k=1$. 
\begin{lemma}
For $k=1$,
\begin{equation*}
\frac{\mathrm{d}}{\mathrm{d}t}\mathrm{E}_k(t)\lesssim \norm{u}_{H^2(\mathbb{R})}\norm{u}_{H^1(\mathbb{R})}^3.
\end{equation*}
\end{lemma}
\begin{proof}
When $k=1$, we know from Lemma \ref{energy-first-decomp} that
\begin{equation*}
\frac{1}{2}\frac{\mathrm{d}}{\mathrm{d}t}\mathrm{E}_k(t)=F_1+G_1,
\end{equation*}
and we know from Lemma \ref{F_1-G_1} that $G_1=2F_1+\mathcal{O}(\norm{u}_{H^3(\mathbb{R})}\norm{u}_{H^1(\mathbb{R})}^3)$. Hence, it only remains to estimate $F_1$, and it is easy to establish, using \eqref{est-2}, that
\begin{equation*}
\abs{F_1}\lesssim \norm{u}_{H^2(\mathbb{R})}\norm{u}_{H^1(\mathbb{R})}^3.
\end{equation*}
\end{proof}
We next estimate the terms $F_j$, $G_j$, $j=2,3\ldots, k-1$ when $k\geq 2$. 
\begin{lemma}\label{F_j-G_j-est}
For $k\geq 2$ and $j=2,3\ldots, k-1$,
\begin{equation*}
\abs{F_j}+\abs{G_j}\lesssim \norm{u}_{H^k(\mathbb{R})}^4.
\end{equation*}
\end{lemma}
\begin{proof}
First we consider $F_j$ and arguing as before, it is only necessary to consider the high frequencies. 
\begin{align}
&\abs{\int_{(\xi-\eta)^2+\eta^2\geq 1}m(\xi-\eta,\eta)(\mathrm{i}\eta)^{k+1-j}\hat{u}(\eta)(\mathrm{i}(\xi-\eta))^j\hat{u}(\xi-\eta)\overline{(\mathrm{i}\xi)^k\FF(u^2)(\xi)}\ \mathrm{d}\eta\ \mathrm{d}\xi}\nonumber\\
&\quad \lesssim \norm{u}_{H^{k-j+1}(\mathbb{R})}\norm{u}_{H^{j+1}(\mathbb{R})}\norm{u}_{H^k(\mathbb{R})}^2+\norm{u}_{H^{j-1}(\mathbb{R})}\norm{u}_{H^{k+2-j}(\mathbb{R})}\norm{u}_{H^k(\mathbb{R})}^2\nonumber\\
&\quad \lesssim \norm{u}_{H^k(\mathbb{R})}^4,\label{estimate-F_j}
\end{align}
where we used \eqref{est-2} in the first estimate. 

For $G_j$ we first note that
\begin{align*}
G_j&=-\int_{\mathbb{R}^2}(1+\abs{\xi}^\alpha)m(\xi-\eta,\eta)(1+\abs{\eta}^\alpha)^{-1}(\mathrm{i}\eta)^{k+1-j}\FF(u^2)(\eta)(\mathrm{i}(\xi-\eta))^j\hat{u}(\xi-\eta)\overline{(\mathrm{i}\xi)^k\hat{u}(\xi)}\ \mathrm{d}\eta\ \mathrm{d}\xi\\
&= -\int_{\mathbb{R}^2}(1+\abs{\eta}^\alpha)m(\eta-\xi,\xi)(1+\abs{\xi}^\alpha)^{-1}(\mathrm{i}\xi)^{k+1-j}\FF(u^2)(\xi)(\mathrm{i}(\eta-\xi))^j\hat{u}(\eta-\xi)\overline{(\mathrm{i}\eta)^k\hat{u}(\eta)}\ \mathrm{d}\eta\ \mathrm{d}\xi\\
&=-\int_{\mathbb{R}^2}m(\xi-\eta,\eta)(\mathrm{i}\xi)^{k-j}\FF(u^2)(\xi)(\mathrm{i}(\eta-\xi))^j\hat{u}(\eta-\xi)\overline{(\mathrm{i}\eta)^{k+1}\hat{u}(\eta)}\ \mathrm{d}\eta\ \mathrm{d}\xi\\
&=-\int_{\mathbb{R}^2}m(\xi-\eta,\eta)(\mathrm{i}\eta)^{k+1}\hat{u}(\eta)(\mathrm{i}(\xi-\eta))^j\hat{u}(\xi-\eta)\overline{(\mathrm{i}\xi)^{k-j}\FF(u^2)(\xi)}\ \mathrm{d}\eta\ \mathrm{d}\xi\\
&=\int_{\mathbb{R}^2}m(\xi-\eta,\eta)(\mathrm{i}\eta)^k\hat{u}(\eta)(\mathrm{i}(\xi-\eta))^{j+1}\hat{u}(\xi-\eta)\overline{(\mathrm{i}\xi)^{k-j}\FF(u^2)(\xi)}\ \mathrm{d}\eta\ \mathrm{d}\xi\\
&\quad +\int_{\mathbb{R}^2}m(\xi-\eta,\eta)(\mathrm{i}\eta)^k\hat{u}(\eta)(\mathrm{i}(\xi-\eta))^j\hat{u}(\xi-\eta)\overline{(\mathrm{i}\xi)^{k-j+1}\FF(u^2)(\xi)}\ \mathrm{d}\eta\ \mathrm{d}\xi
\end{align*}
and these integrals can be bounded by $\norm{u}_{H^k}^4$, using \eqref{est-xi} and arguing as in \eqref{estimate-F_j}.
\end{proof}
We continue by decomposing $F_1,F_k,G_1$ further:
\begin{align*}
F_k&=F_1\\
&=\frac{1}{2}\langle \partial_x^k(u^2)u,P(\partial_x^ku,\partial_xu)\rangle\\
&=\underbrace{\langle u\partial_x^ku,P(\partial_x^ku,\partial_xu)\rangle}_{=:F_{1,0}}+\sum_{l=1}^{k-1}c_{k,l}\frac{1}{2}\underbrace{\langle\partial_x^{k-l}u\partial_x^lu,P(\partial_x^ku,\partial_xu)\rangle}_{=:F_{1,l}},
\end{align*}
\begin{align*}
G_1&=-\langle\partial_x^ku,(1+\mathrm{D}^\alpha)P((1+\mathrm{D}^\alpha)^{-1}\partial_x^k(u^2),\partial_xu)\rangle\\
&=\underbrace{-2\langle\partial_x^ku,(1+\mathrm{D}^\alpha)P((1+\mathrm{D}^\alpha)^{-1}\partial_x(u\partial_x^{k-1}u),\partial_xu)\rangle}_{=:G_{1,0}}\\
&\quad +\sum_{l=1}^{k-2}c_{k,l}\left(\underbrace{-\langle\partial_x^ku,(1+\mathrm{D}^\alpha)P((1+\mathrm{D}^\alpha)^{-1}\partial_x(\partial_x^{k-1-l}u\partial_x^lu),\partial_xu)\rangle}_{=:G_{1,l}}\right).
\end{align*}
We start by estimating $F_{1,l}$ and $G_{1,l}$, for $l=1,2,\ldots, k-1$.
\begin{lemma}\label{F_1l-G_1l-est}
For $k\geq 2$ and $l=1,2,\ldots, k-1$,
\begin{equation*}
\abs{F_{1,l}}+\abs{G_{1,l}}\lesssim \norm{u}_{H^k(\mathbb{R})}^4.
\end{equation*}
\end{lemma}
\begin{proof}
This inequality can be established using the same techniques as in the proof of Lemma \ref{F_j-G_j-est}.
\end{proof}
Combining Lemmata \ref{F_j-G_j-est}, \ref{F_1l-G_1l-est} we immediately get, for $k\geq 2$
\begin{equation}\label{en-est2}
\frac{\mathrm{d}}{\mathrm{d}t}\mathrm{E}_k(t)\lesssim k(F_{1,0}+G_{1,0})+F_{1,0}+G_{k}+\norm{u}_{H^k(\mathbb{R})}^4,
\end{equation}
Hence it remains to estimate $F_{1,0}$, $G_{1,0}$ and $G_k$. The first two terms cannot be estimated in a straightforward way, due to the fact that there are to many derivatives on $u$. The idea is therefore to first consider $2F_{1,0}+G_{1,0}$. The reason for having a factor $2$ in front of $F_{1,0}$ is due to the $2$ appearing in the definition of $G_{1,0}$. By considering $2F_{1,0}+G_{1,0}$ we get a good commutator, in the sense that derivatives are canceled, that we are able to estimate. Moreover, the following lemma ensures that if $\abs{2F_{1,0}+G_{1,0}}\lesssim \norm{u}_{H^{2+\frac{\alpha}{2}}}^2\norm{u}_{H^{k+\frac{\alpha}{2}}}^2$, then $\abs{F_{1,0}},\abs{G_{1,0}}\lesssim \norm{u}_{H^{2+\frac{\alpha}{2}}}^2\norm{u}_{H^{k+\frac{\alpha}{2}}}^2$.
\begin{lemma}\label{F_10-G_10}
For $k\geq 1$,
\begin{equation*}
G_{1,0}=2F_{1,0}+\mathcal{O}(\norm{u}_{H^2}^2\norm{u}_{H^k}^2).
\end{equation*}
\end{lemma}
\begin{proof}
The proof is very similar to the proof of Lemma \ref{F_1-G_1}, and is therefore omitted.
\end{proof}
We point out here that Lemma \ref{F_10-G_10} corresponds to \cite[Lemma 4.7]{EW2018}, but is more general, since we do not make any restriction on the domain of integration.
We now proceed to estimate $2F_{1,0}+G_{1,0}$, using the same strategy as in \cite{EW2018}. We first rewrite both $F_{1,0}$ and $G_{1,0}$.
\begin{align*} 
F_{1,0}&=\int_{\mathbb{R}^3}m(\eta-\sigma,\sigma)(\mathrm{i}\sigma)^k\mathrm{i}(\eta-\sigma)\overline{(\mathrm{i}\xi)^k}\ \mathrm{d}\mathrm{Q}(u),\\
G_{1,0}&=-2\int_{\mathbb{R}^3}\frac{1+\abs{\xi}^\alpha}{1+\abs{\eta}^\alpha}m(\xi-\eta,\eta)\mathrm{i}\eta(\mathrm{i}\sigma)^{k-1}\mathrm{i}(\xi-\eta)\overline{(\mathrm{i}\xi)^k}\ \mathrm{d}\mathrm{Q}(u),
\end{align*}
where
\begin{equation*}
\mathrm{d}\mathrm{Q}(u)=\hat{u}(\eta-\sigma)\hat{u}(\xi-\eta)\hat{u}(\sigma)\overline{\hat{u}(\xi)}.
\end{equation*}
The next step is to decompose $\mathbb{R}^3$, but before doing this we  make the change of variable $\eta\mapsto \eta-\xi+\sigma$ in $F_{1,0}$, so that
\begin{equation*}
F_{1,0}=\int_{\mathbb{R}^3}m(\xi-\eta,\sigma)(\mathrm{i}\sigma)^k\mathrm{i}(\xi-\eta)\overline{(\mathrm{i}\xi)^k}\ \mathrm{d}\mathrm{Q}(u).
\end{equation*}
This differs from the approach taken in \cite{EW2018} where the change of variables is performed after the decomposition. The benefit of doing it before is that there is no need for the technical lemma \cite[Lemma 4.9]{EW2018}, however the downside is that the proof of Lemma \ref{final-est} becomes slightly more involved.
Next we decompose $\mathbb{R}^3$, starting with the set
\begin{equation*}
\mathcal{A}_1:=\{(\xi,\eta,\sigma)\in \mathbb{R}^3\colon\min\{\abs{\xi},\abs{\eta},\abs{\sigma}\}<1\}.
\end{equation*}
For convenience we introduce the notation $\mathcal{A}_1F_{1,0}$, $\mathcal{A}_1G_{1,0}$ to indicate that the integrals are taken over $\mathcal{A}_1$. The elements of $\mathcal{A}_1$ satisfy
\begin{equation}\label{transfer-der}
\abs{\xi}+\abs{\eta}+\abs{\sigma}\lesssim 1+\abs{\xi-\eta}+\abs{\eta-\sigma},
\end{equation}
and this allows us to move factors of $\xi$, $\eta$ and $\sigma$ to $\xi-\eta$ and $\eta-\sigma$ which makes it possible to estimate $\mathcal{A}_1F_{1,0}$, $\mathcal{A}_1G_{1,0}$ directly.
\begin{lemma}\label{a01-est}
The integrals $\mathcal{A}_1F_{1,0}$, $\mathcal{A}_1G_{1,0}$ satisfy
\begin{equation*}
\abs{\mathcal{A}_1F_{1,0}}\lesssim \norm{u}_{H^2(\mathbb{R})}^2\norm{u}_{H^k(\mathbb{R})}^2,\ \quad \abs{\mathcal{A}_1G_{1,0}}\lesssim \norm{u}_{H^{2+\frac{\alpha}{2}}(\mathbb{R})}^2\norm{u}_{H^{k+\frac{\alpha}{2}}(\mathbb{R})}^2,\ \quad k\geq 2.
\end{equation*}
\end{lemma}
\begin{proof}
Using \eqref{transfer-der} we can transfer factors of $\xi,\ \eta$ and $\sigma$ to $\eta$ and $\sigma$ to $\xi-\eta$ and $\eta-\sigma$ as needed, allowing us to estimate $\mathcal{A}_1F_{1,0}$, $\mathcal{A}_1G_{1,0}$ using Proposition \ref{prop-m}.
\end{proof}
The next step is to estimate $2\mathcal{A}_1^cF_{1,0}+\mathcal{A}_1^cG_{1,0}$, and this is achieved by decomposing $\mathcal{A}_1^c$ further. Indeed, let
\begin{equation*}
\mathcal{A}_2:=\{(\xi,\eta,\sigma)\in\mathcal{A}_1^c\colon\frac{1}{10}\abs{z_2}<\abs{z_1-z_2}+\abs{z_2-z_3}, \text{for some choice of } z_j=\xi,\eta,\sigma\}
\end{equation*}
and write
\begin{equation*}
\mathcal{A}_1^c=\mathcal{A}_2\cup\mathcal{A}_2^c.
\end{equation*}
It is straightforward to obtain estimates for $\mathcal{A}_2F_{1,0}$ and $\mathcal{A}_2G_{1,0}$.
\begin{lemma}
The integrals $\mathcal{A}_2F_{1,0}$, $\mathcal{A}_2G_{1,0}$ satisfy
\begin{equation*}
\abs{\mathcal{A}_2F_{1,0}}\lesssim \norm{u}_{H^2(\mathbb{R})}^2\norm{u}_{H^k(\mathbb{R})}^2,\quad \abs{\mathcal{A}_2G_{1,0}}\lesssim \norm{u}_{H^{2+\frac{\alpha}{2}}(\mathbb{R})}^2\norm{u}_{H^{k+\frac{\alpha}{2}}(\mathbb{R})}^2, \quad k\geq 2.
\end{equation*}
\end{lemma} 
\begin{proof}
The idea here is precisely the same as in the proof of Lemma \ref{a01-est}. The defining property of $\mathcal{A}_2$ allows us to transfer factors of $\xi,\eta,\sigma$ to $\xi-\eta$ or $\eta-\sigma$, and the desired estimates are then obtained using Proposition \ref{prop-m}.
\end{proof}
The final task is therefore to estimate $2\mathcal{A}_2^cF_{1,0}+\mathcal{A}_2^cG_{1,0}$, and we first note that
\begin{align*}
2\mathcal{A}_2^cF_{1,0}+\mathcal{A}_2^cG_{1,0}&=2\mathrm{i}\int_{\mathcal{A}_2^c}\left[m(\xi-\eta,\sigma)\sigma-\frac{1+\abs{\xi}^\alpha}{1+\abs{\eta}^\alpha}m(\xi-\eta,\eta)\eta\right](\sigma)^{k-1}(\xi-\eta)\xi^k\ \mathrm{d}\mathrm{Q}(u)\\
&=2\mathrm{i}\int_{\mathcal{A}_2^c}\left[m(\xi-\eta,\sigma)-\frac{1+\abs{\xi}^\alpha}{1+\abs{\eta}^\alpha}m(\xi-\eta,\eta)\right]\sigma^k(\xi-\eta)\xi^k\ \mathrm{d}\mathrm{Q}(u)\\
&\quad -2\mathrm{i}\int_{\mathcal{A}_2^c}\frac{1+\abs{\xi}^\alpha}{1+\abs{\eta}^\alpha}m(\xi-\eta,\eta)(\eta-\sigma)\sigma^{k-1}(\xi-\eta)\xi^k\ \mathrm{d}\mathrm{Q}(u),
\end{align*}
where the last integral can be estimated by $\norm{u}_{H^{2+\frac{\alpha}{2}}}^2\norm{u}_{H^{k+\frac{\alpha}{2}}}^2$, using proposition \ref{prop-m}. Hence,
\begin{equation*}
2\mathcal{A}_2^cF_{1,0}+\mathcal{A}_2^cG_{1,0}=2I+\mathcal{O}(\norm{u}_{H^{2+\frac{\alpha}{2}}}^2\norm{u}_{H^{k+\frac{\alpha}{2}}}^2),
\end{equation*}
where 
\begin{equation*}
I:=\mathrm{i}\int_{\mathcal{A}_2^c}\left[m(\xi-\eta,\sigma)-\frac{1+\abs{\xi}^\alpha}{1+\abs{\eta}^\alpha}m(\xi-\eta,\eta)\right]\sigma^k(\xi-\eta)\xi^k\ \mathrm{d}\mathrm{Q}(u).
\end{equation*}
In order to estimate $I$ we first discuss some properties of $\mathcal{A}_2^c$. By definition, elements $(\xi,\eta,\sigma)\in\mathcal{A}_2^c$ satisfy 
\begin{equation*}
\frac{1}{10}\geq  \abs{\frac{\xi}{\eta}-1}+\abs{1-\frac{\sigma}{\eta}},
\end{equation*}
which implies that $\text{sgn}(\xi)=\text{sgn}(\eta)=\text{sgn}(\sigma)$. This will allow us to integrate over $\mathcal{A}_{2,+}^c:=\{(\xi,\eta,\sigma)\in\mathcal{A}_2^c\ :\ \xi,\eta,\sigma\geq 1\}$ instead, since $I=2\mathcal{A}_{2,+}^cI$. Moreover, for elements $(\xi,\eta,\sigma)\in\mathcal{A}_{2,+}^c$ we have that
\begin{equation}\label{xi-sigma-eta-rel}
\xi=(1+\mu)\eta,\quad \sigma=(1+\nu)\eta,
\end{equation}
for $\abs{\mu},\abs{\nu}\leq \frac{1}{10}$. We are now ready to estimate $\mathcal{A}_{2,+}^cI$. 
\begin{lemma}\label{final-est}
The integral $\mathcal{A}_{2,+}^cI$ satisfies
\begin{equation*}
\abs{\mathcal{A}_{2,+}^cI}\lesssim \norm{u}_{H^{2+\frac{\alpha}{2}}(\mathbb{R})}^2\norm{u}_{H^{k+\frac{\alpha}{2}}(\mathbb{R})}^2,\quad k\geq 2
\end{equation*}
\end{lemma}
\begin{proof}
We have that
\begin{equation*}
\mathcal{A}_{2,+}^cI=\mathrm{i}\int_{\mathcal{A}_{2,+}^c}\left[m(\xi-\eta,\sigma)-\frac{1+\abs{\xi}^\alpha}{1+\abs{\eta}^\alpha}m(\xi-\eta,\eta)\right]\sigma^k(\xi-\eta)\xi^k\ \mathrm{d}\mathrm{Q}(u),
\end{equation*}
and
\begin{align*}
N(\xi,\eta,\sigma):&=\left[m(\xi-\eta,\sigma)-\frac{1+\abs{\xi}^\alpha}{1+\abs{\eta}^\alpha}m(\xi-\eta,\eta)\right]\\
&=\frac{m(\xi-\eta,\sigma)m(\xi-\eta,\eta)}{\xi(\xi-\eta+\sigma)(1+\abs{\xi}^\alpha)(1+\abs{\sigma}^\alpha)(1+\abs{\xi-\eta}^\alpha)}\tilde{N}(\xi,\eta,\sigma),
\end{align*}
where
\begin{align*}
\tilde{N}(\xi,\eta,\sigma):&=(\xi-\eta+\sigma)(1+\abs{\sigma}^\alpha)\left[\xi(1+\abs{\eta}^\alpha)(\abs{\xi-\eta}^\alpha\abs{\xi}^\alpha)-\eta(1+\abs{\xi}^\alpha)(\abs{\xi-\eta}^\alpha-\abs{\eta}^\alpha)\right]\\
&\quad-\xi(1+\abs{\xi}^\alpha)\left[(\xi-\eta+\sigma)(1+\abs{\sigma}^\alpha)(\abs{\xi-\eta}^\alpha-\abs{\xi-\eta+\sigma}^\alpha)\right.\\
&\left.\quad-\sigma(1+\abs{\xi-\eta+\sigma}^\alpha)(\abs{\xi-\eta}^\alpha-\abs{\sigma}^\alpha)\right]
\end{align*}
In $\mathcal{A}_{2,+}^c$ we have, due to \eqref{xi-sigma-eta-rel}, that
\begin{equation}\label{A2+-prop}
\xi\simeq\eta\simeq\sigma\simeq\xi-\eta+\sigma\gtrsim 1.
\end{equation}
Using \eqref{A2+-prop} together with \eqref{est-1}, we find that
\begin{align}
&\frac{\abs{m(\xi-\eta,\sigma)m(\xi-\eta,\eta)}}{\xi(\xi-\eta+\sigma)(1+\abs{\xi}^\alpha)(1+\abs{\sigma}^\alpha)(1+\abs{\xi-\eta}^\alpha)}\nonumber\\
&\lesssim \frac{\left(\frac{\abs{\xi-\eta}}{\sigma}+\frac{\sigma}{\abs{\xi-\eta}}\right)\left(\frac{\xi-\eta}{\eta}+\frac{\eta}{\abs{\xi-\eta}}\right)}{\xi(\xi-\eta+\sigma)(1+\abs{\xi}^\alpha)(1+\abs{\sigma}^\alpha)(1+\abs{\xi-\eta}^\alpha)}\nonumber\\
&\lesssim \frac{\left(\frac{\abs{\xi-\eta}^2}{\xi^4}+\frac{1}{\xi^2}+\frac{1}{\abs{\xi-\eta}^2}\right)}{(1+\abs{\xi}^\alpha)(1+\abs{\sigma}^\alpha)(1+\abs{\xi-\eta}^\alpha)}.\label{firstN-est}
\end{align}
Next, using \eqref{xi-sigma-eta-rel} and expanding in Taylor series, we find that
\begin{align*}
\tilde{N}(\xi,\eta,\sigma)&=(1+\nu+\mu)\eta(1+(1+\nu)^\alpha\eta^\alpha)\big[(1+\mu)\eta(1+\eta^\alpha)(\mu^\alpha\eta^\alpha-(1+\mu)^\alpha\eta^\alpha)\\
&\quad-\eta(1+(1+\mu)^\alpha\eta^\alpha)(\mu^\alpha\eta^\alpha-\eta^\alpha)\big]-(1+\mu)\eta(1+(1+\mu)^\alpha\eta^\alpha)\\
&\quad\times\big[(1+\nu+\mu)\eta(1+(1+\nu)^\alpha\eta^\alpha)(\mu^\alpha\eta^\alpha-(1+\nu+\mu)^\alpha\eta^\alpha)\\
&\quad -(1+\nu)\eta(1+(1+\nu+\mu)^\alpha\eta^\alpha)(\mu^\alpha\eta^\alpha-(1+\nu)^\alpha\eta^\alpha)\big]\\
&=\eta^{2+\alpha}(1+\nu+\mu)(1+(1+\nu)^\alpha\eta^\alpha)\big[\mu(\mu^\alpha-1-\alpha)+\eta^\alpha\mu(\mu^\alpha-1-\alpha\mu^\alpha)\big]\\
&\quad-\eta^{2+\alpha}(1+\mu)(1+(1+\mu)^\alpha\eta^\alpha)\big[\mu(\mu^\alpha-(1+\alpha)(1+\nu)^\alpha)\\
&\quad+\eta^\alpha\mu(\mu^\alpha-(1+\alpha\mu^\alpha)(1+\nu)^\alpha)\big]+\eta^{2+\alpha}(1+\eta^\alpha)^2\mathcal{O}(\mu^2)\\
&=\eta^{2+\alpha}(1+\eta^\alpha)\big[\mu(\mu^\alpha-1-\alpha)+\eta^\alpha\mu(\mu^\alpha-1-\alpha\mu^\alpha)-\mu(\mu^\alpha-(1+\alpha)(1+\nu)^\alpha)\\
&\quad-\eta^\alpha\mu(\mu^\alpha-(1+\alpha\mu^\alpha)(1+\nu)^\alpha)\big]+\eta^{2+\alpha}(1+\eta^\alpha)^2\mathcal{O}(\mu^2+\mu\nu)\\
&=\eta^{2+\alpha}(1+\eta^\alpha)^2\mathcal{O}(\mu^2+\mu\nu).
\end{align*}
Hence,
\begin{equation}\label{tildeN-est}
\abs{\tilde{N}(\xi,\eta,\sigma)}\lesssim \eta^{2+\alpha}(1+\eta^\alpha)^2\abs{\mu}(\abs{\mu}+\abs{\nu})=\eta^\alpha(1+\eta^\alpha)^2\abs{\xi-\eta}(\abs{\xi-\eta}+\abs{\eta-\sigma}).
\end{equation}
Using \eqref{firstN-est} together with \eqref{tildeN-est} we immediately get that
\begin{align*}
\abs{N(\xi,\eta,\sigma)}&\lesssim \eta^\alpha \abs{\xi-\eta}(\abs{\xi-\eta}+\abs{\eta-\sigma})\left(\frac{\abs{\xi-\eta}^2}{\xi^4}+\frac{1}{\xi^2}+\frac{1}{\abs{\xi-\eta}^2}\right)\\
&\leq \eta^\alpha\left(\frac{\abs{\mu}^3(\abs{\mu}+\abs{\nu})}{(1+\mu)^4}+\frac{\abs{\mu}(\abs{\mu}+\abs{\nu})}{(1+\mu)^2}+1+\frac{\abs{\eta-\sigma}}{\abs{\xi-\eta}}\right)\\
&\lesssim \eta^\alpha\left(1+\frac{\abs{\eta-\sigma}}{\abs{\xi-\eta}}\right) 
\end{align*}
from which it follows that
\begin{align*}
\abs{\mathcal{A}_{2,+}^cI}&\lesssim\int_{\mathcal{A}_{2,+}^c}\eta^\alpha\left(1+\frac{\abs{\eta-\sigma}}{\abs{\xi-\eta}}\right)\abs{\xi-\eta}\sigma^k\xi^k\ \mathrm{d}\mathrm{Q}(u)\\
&\lesssim \norm{u}_{H^{2+\frac{\alpha}{2}}(\mathbb{R})}^2\norm{u}_{H^{k+\frac{\alpha}{2}}(\mathbb{R})}^2.
\end{align*}
\end{proof}

Finally we estimate $G_k$.
\begin{lemma}\label{G_k-est}
The integral $G_k$ satisfies
\begin{equation*}
\abs{G_k}\lesssim \norm{u}_{H^{2+\frac{\alpha}{2}}(\mathbb{R})}^2\norm{u}_{H^{k+\frac{\alpha}{2}}(\mathbb{R})}^2,\quad k\geq 2
\end{equation*}
\end{lemma}
\begin{proof}
First note that
\begin{align}
G_k&=-\langle \partial_x^ku,(1+\mathrm{D}^\alpha)P((1+\mathrm{D}^\alpha)^{-1}\partial_x(u^2),\partial_x^ku)\rangle\nonumber\\
&=-\int_{\mathbb{R}^2}(1+\abs{\xi}^\alpha)m(\xi-\eta,\eta)(\mathrm{i}\eta)^k\hat{u}(\eta)(1+\abs{\xi-\eta}^\alpha)^{-1}\mathrm{i}(\xi-\eta)\mathcal{F}(u^2)(\xi-\eta)\overline{(\mathrm{i}\xi)^k\hat{u}(\xi)}\ \mathrm{d}\eta\mathrm{d}\xi\nonumber\\
&=\int_{\mathbb{R}^2}(1+\abs{\xi}^\alpha)m(\xi-\eta,\eta)(\mathrm{i}\eta)^k\hat{u}(\eta)(1+\abs{\xi-\eta}^\alpha)^{-1}(\mathrm{i}(\xi-\eta))^2\mathcal{F}(u^2)(\xi-\eta)\overline{(\mathrm{i}\xi)^{k-1}\hat{u}(\xi)}\ \mathrm{d}\eta\mathrm{d}\xi\nonumber\\
&\quad+\int_{\mathbb{R}^2}(1+\abs{\xi}^\alpha)m(\xi-\eta,\eta)(\mathrm{i}\eta)^{k+1}\hat{u}(\eta)(1+\abs{\xi-\eta}^\alpha)^{-1}\mathrm{i}(\xi-\eta)\mathcal{F}(u^2)(\xi-\eta)\overline{(\mathrm{i}\xi)^{k-1}\hat{u}(\xi)}\ \mathrm{d}\eta\mathrm{d}\xi,\label{G_k-decomp}
\end{align}
and using \eqref{functional-rel}, we have that
\begin{align*}
&\int_{\mathbb{R}^2}(1+\abs{\xi}^\alpha)m(\xi-\eta,\eta)(\mathrm{i}\eta)^{k+1}\hat{u}(\eta)(1+\abs{\xi-\eta}^\alpha)^{-1}\mathrm{i}(\xi-\eta)\mathcal{F}(u^2)(\xi-\eta)\overline{(\mathrm{i}\xi)^{k-1}\hat{u}(\xi)}\ \mathrm{d}\eta\mathrm{d}\xi\\
&=\int_{\mathbb{R}^2}(1+\abs{\eta}^\alpha)m(\eta-\xi,\xi)(\mathrm{i}\xi)^{k+1}\hat{u}(\xi)(1+\abs{\xi-\eta}^\alpha)^{-1}\mathrm{i}(\eta-\xi)\mathcal{F}(u^2)(\eta-\xi)\overline{(\mathrm{i}\eta)^{k-1}\hat{u}(\eta)}\ \mathrm{d}\eta\mathrm{d}\xi\\
&=-\int_{\mathbb{R}^2}(1+\abs{\xi}^\alpha)m(\xi-\eta,\eta)\mathrm{i}\eta(\mathrm{i}\xi)^k\hat{u}(\xi)(1+\abs{\xi-\eta}^\alpha)^{-1}\mathrm{i}(\eta-\xi)\mathcal{F}(u^2)(\eta-\xi)\overline{(\mathrm{i}\eta)^{k-1}\hat{u}(\eta)}\ \mathrm{d}\eta\mathrm{d}\xi\\
&=\int_{\mathbb{R}^2}(1+\abs{\xi}^\alpha)m(\xi-\eta,\eta)(\mathrm{i}\eta)^k\hat{u}(\eta)(1+\abs{\xi-\eta}^\alpha)^{-1}\mathrm{i}(\xi-\eta)\mathcal{F}(u^2)(\xi-\eta)\overline{(\mathrm{i}\xi)^k\hat{u}(\xi)}\ \mathrm{d}\eta\mathrm{d}\xi\\
&=-G_k.
\end{align*}
From \eqref{G_k-decomp} we then get that
\begin{equation*}
G_k=\frac{1}{2}\int_{\mathbb{R}^2}(1+\abs{\xi}^\alpha)m(\xi-\eta,\eta)(\mathrm{i}\eta)^k\hat{u}(\eta)(1+\abs{\xi-\eta}^\alpha)^{-1}(\mathrm{i}(\xi-\eta))^2\mathcal{F}(u^2)(\xi-\eta)\overline{(\mathrm{i}\xi)^{k-1}\hat{u}(\xi)}\ \mathrm{d}\eta\mathrm{d}\xi,
\end{equation*}
and using \eqref{est-xi} together with arguments similar to those used in the proof Lemma \ref{F_1-G_1}, it is easy to see that the absolute value of above integral is bounded above by $\norm{u}_{H^{2+\frac{\alpha}{2}}}^2\norm{u}_{H^{k+\frac{\alpha}{2}}(\mathbb{R})}^2$.
\end{proof}

Proposition \ref{en-estimates} now follows by combining Lemmata \ref{energy-first-decomp}--\ref{G_k-est}. 

\normalsize
\bibliographystyle{plain}

\end{document}